\documentclass{amsart}

\usepackage{amssymb,amscd,amsthm,latexsym, amsmath}
\usepackage[all]{xy}

\newtheorem{Lemma}{Lemma}[section]

\newtheorem{theorem}[Lemma]{Theorem}
\newtheorem{lemma}[Lemma]{Lemma}
\newtheorem{proposition}[Lemma]{Proposition}
\newtheorem{corollary}[Lemma]{Corollary}
\newtheorem{definition}[Lemma]{Definition}

\def\into{ \rightarrowtail }
\def\onto{ \twoheadrightarrow }
\def\splito{ \rightleftarrows }
\newdir{ >}{{}*!/-8pt/@{>}}
\def\EE{ \mathbb{E} }
\def\CC{ \mathbb{C} }

\def\VV{ \mathbb{V} }

\def\TT{ \mathbb{T} }
\def\AAA{ \mathbb{A} }

\def\trio{ \triangleright}

\def\cleft{\hbox{[\kern-.16em\hbox{[}}}
\def\cright{\hbox{]\kern-.16em\hbox{]}}}

\newcommand{\Ker}{ \ensuremath{\mathrm{Ker}} }

\newcommand{\Eq}{ \ensuremath{\mathrm{Equ}}}
\newcommand{\Rng}{\mathsf{Rng}}
\newcommand{\Alg}{\mathsf{Alg}}
\newcommand{\Gp}{\mathsf{Gp}}
\newcommand{\DiGp}{\mathsf{DiGp}}
\newcommand{\Aut}{\mathsf{Aut}}
\newcommand{\SKB}{\mathsf{SKB}}
\newcommand{\Lie}{\mathsf{Lie}}
\newcommand{\Imm}{\mathsf{Im}}
\newcommand{\Equ}{\mathsf{Equ}}
\newcommand{\Mon}{\mathsf{Mon}}
\newcommand{\Norm}{\mathsf{Norm}}
\newcommand{\id}{\mbox{\rm id}}

\begin{document}
		\title{Aspects of the Category $\SKB$ of Skew Braces}

	\author{Dominique Bourn}
	\address{Univ. Littoral C\^ote d'Opale, UR 2597, LMPA,
Laboratoire de Math\'ematiques Pures et Appliqu\'ees Joseph Liouville,
F-62100 Calais, France}
\email{bourn@univ-littoral.fr}
\author{Alberto Facchini}
	\address{Dipartimento di Matematica ``Tullio Levi-Civita'', Universit\`a di 
Padova, 35121 Padova, Italy}
 \email{facchini@math.unipd.it}
\thanks{The second author is partially supported by Ministero dell'Istruzione, dell'Universit\`a e della Ricerca (Progetto di ricerca di rilevante interesse nazionale ``Categories, Algebras: Ring-Theoretical and Homological Approaches (CARTHA)''), Fondazione Cariverona (Research project ``Reducing complexity in algebra, logic, combinatorics - REDCOM'' within the framework of the programme Ricerca Scientifica di Eccellenza 2018), and the Department of Mathematics ``Tullio Levi-Civita'' of the University of Padua (Research programme DOR1828909 ``Anelli e categorie di moduli'').}
\author{Mara Pompili}
\address{Dipartimento di Matematica ``Tullio Levi-Civita'', Universit\`a di 
Padova, 35121 Padova, Italy}
 \email{mara.pompili@studenti.unipd.it}
   \begin{abstract} We examine the pointed protomodular category $\SKB$ of left skew braces. We study the notion of commutator of ideals in a left skew brace. Notice that in the literature, ``product'' of ideals of skew braces is often considered. We show that Huq=Smith for left skew braces. Finally, we give a set of generators for the commutator of two ideals, and prove that every ideal of a left skew brace has a centralizer.
\end{abstract}

    \maketitle

	\date{}
	
	\maketitle
	
	\section*{Introduction}
	
	Braces appear in connections to the study of set-theoretic solutions of the Yang-Baxter equation.
	A {\em set-theoretic solution of the Yang-Baxter equation} is a pair $(X,r)$, where $X$ is a set, $r\colon X\times X\to X\times X$ is a bijection, and $(r\times\id)(\id\times r)(r\times\id)=(\id\times r)(r\times\id)(\id\times r)$
	\cite{23}.  Set-theoretic solutions of the Yang-Baxter equation appear, for instance, in the study of representations of braid groups, and form a category {\bf SYBE}, whose objects are these pairs $(X,r)$, and morphisms $f\colon (X,r)\to(X',r')$ are the mappings $f\colon X\to X'$ that make the diagram
	$$\xymatrix{
		X\times X\ar[r]^{f\times f}\ar[d]^r&X'\times X'\ar[d]_r\\
		X\times X\ar[r]_{f\times f}&X'\times X'}$$ commute.
	
	One way to produce set-theoretic solutions of the Yang-Baxter equation is using left skew braces.
	
	\medskip
	 
\noindent\textbf{Definition} 
	{\rm \cite{GV} A {\sl (left) skew brace} is
		a triple $(A, *,\circ)$, where $(A,*) $ and $(A,\circ)$ are groups (not necessarily abelian) such that \begin{equation}a\circ (b * c) = (a\circ b)*a^{-*}* ( a\circ c)\label{lsb}\tag{B}\end{equation}
		for every $a,b,c\in A$. Here $a^{-*}$ denotes the inverse of $a$ in the group $(A,*)$. The inverse of $a$ in the group $(A,\circ)$ will be denoted by $a^{-\circ}$.}
		
A brace is sometimes seen as an algebraic structure similar to that of a ring, with distributivity warped in some sense. But a better description of a brace is probably that of an algebraic structure with two group structures out of phase with each other.
	
\medskip

For every left skew brace $(A, *,\circ)$, the mapping $$r \colon A\times A \to A \times A,\quad r(x,y) = (x^{-*} *(x\circ y),(x^{-*} *(x\circ y))^{-\circ}\circ x\circ y),$$
is a non-degenerate set-theoretic solution of the Yang-Baxter equation (\cite[Theorem~3.1]{GV} and \cite[p.~96]{KSV}). Here ``non-degenerate'' means that the mappings $\pi_1r(x_0,-)\colon A\to A$ and $\pi_2r(-,y_0)\colon A\to A$  are bijections for every $x_0\in A$ and every $y_0\in A$.

The simplest examples of left skew braces are:

(1) For any associative ring $(R,+, \cdot)$, the Jacobson radical $(J(R),+, \circ)$, where $\circ$ is the operation on $J(R)$ defined by $x\circ y=xy+x+y$ for every $x,y\in J(R)$.

(2) For any group $(G,*)$, the left skew braces $(G,*,*)$ and $(G,*,*^{op})$.

Several non-trivial examples of skew braces can be found in \cite{SV}. A complete classification of braces of low cardinality has been obtained via computer \cite{KSV}.

A homomorphism of skew braces is a mapping which is a group homomorphism for both the operations. This defines the category $\SKB$ of skew braces. 

From \cite{GV}, we know that in a skew brace the units of the two groups coincide. So, $\SKB$ appears as a fully faithful subcategory $\SKB\hookrightarrow \DiGp$  of the category $\DiGp$ of digroups, where a digroup is a triple $(G,*,\circ)$ of a set $G$ endowed with two group structures with same unit. This notion was introduced in \cite{Normal} and devised during discussions between the first author and G. Janelidze. 

There are two forgetful functors $U_i:\DiGp \to \Gp, \; i\in\{0,1\},$  associating respectively the first and the second group structures. They both reflect isomorphisms. Since $U_0$ is left exact and reflects isomorphisms, it naturally allows the lifting of the protomodular  aspects of the category $\Gp$ of groups to the category $\DiGp$. In turn, the left exact fully faithful embedding $\SKB\hookrightarrow \DiGp$ makes $\SKB$ a pointed protomodular category. The protomodular axiom  was introduced in \cite{B0} in order to extract the essence of the homological constructions and in particular to induce an {\em intrinsic notion of
	exact sequence}.

In this paper, after recalling the basic facts about protomodular categories, we study the ``protomodular aspects'' of left skew braces, in particular in relation to the category of digroups. We study the notion of commutator of ideals in a left skew brace (in the literature, ``product'' of ideals of skew braces is often considered). We show that Huq=Smith for left skew braces. Notice that  Huq${}\ne{}$Smith for digroups and near-rings \cite{Commutators}. We give a set of generators for the commutator of two ideals, and prove that every ideal of a left skew brace has a centralizer.

\section{Basic recalls on protomodular categories}

In this work, any category $\EE$ will be supposed finitely complete, which implies that it has a terminal object $1$. The terminal map from $X$ is denoted $\tau_X: X\to 1$. Given any map $f: X\to Y$, the equivalence relation $R[f]$ on $X$ is produced by the pullback of $f$ along itself. The map $f$ is said to be a {\em regular epimorphism} in $\EE$ when $f$ is the quotient of $R[f]$. When it is the case, we denote it by a double head arrow $X\onto Y$.

\subsection{Pointed protomodular categories}

The category $\EE$ is said to be pointed when the terminal object $1$ is initial as well. Let us recall that a pointed category $\AAA$ is additive if and only if, given any split epimorphism $f: X\splito Y, \; fs=1_Y$, the following downward pullback:
$$
\xymatrix@=20pt{
	\Ker f   \ar@{ >->}[rr]^{k_f} \ar@<-4pt>[d]_{} && X \ar@<-4pt>[d]_{f}  && \\
	1 \ar@{ >->}[rr]_{0_Y} \ar@{ >->}[u]_{0_{K}} & & Y \ar@{ >->}[u]_{s} }
$$
is an upward pushout, namely if and only if $X$ is the direct sum (= coproduct) of $Y$ and $\Ker f$. Let us recall the following:
\begin{definition}{\rm \cite{B0}
	A pointed category $\EE$ is said to be {\em protomodular} when, given any split epimorphism as above, the pair $(k_f,s)$ of monomorphisms is jointly strongly epic.}
\end{definition}
This means that the only subobject $u:U\into X$ containing the pair $(k_f,s)$  of subobjects is, up to isomorphism, $1_X$. It implies that, given any pair $(f,g): X \rightrightarrows Z$ of arrows which are equalized by $k_f$ and $s$, they are necessarily equal (take the equalizer of this pair). Pulling back the split epimorphisms along the initial map $0_Y: 1\into Y$ being a left exact process, the previous definition is equivalent to saying that this process reflects isomorphisms.

The category $\Gp$ of groups is clearly pointed protomodular. This is the case of the category $\Rng$ of rings as well, and more generally, given a commutative ring $R$, of any category  $R$-$\Alg$ of any given kind of $R$-algebras without unit, possibly non-associative. This is in particular the case of the category $R$-$\Lie$ of Lie $R$-algebras. Even for $R$ a non-commutative ring, in which case $R$-algebras have a more complex behaviour (they are usually called $R$-rings, see \cite[p.~36]{Bergman} or \cite[p.~52]{libromio}), one has that the category $R$-$\Rng$ of $R$-rings is pointed protomodular, as can be seen from the fact that the forgetul functor $R$-$\Rng\to Ab$ reflects isomorphisms and $Ab$ is protomodular.

The pointed protomodular axiom  implies that the category $\EE$ shares with the category $\Gp$ of groups the following well-known {\em Five Principles}:\\
(1) a morphism $f$ is a monomorphism if and only if its kernel $\Ker f$ is trivial \cite{B0};\\
(2) any regular epimorphism is the cokernel of its kernel, in other words any regular epimorphism produces an exact sequence, which determines \emph{an intrinsic notion of exact sequences} in $\EE$ \cite{B0};\\
(3) there a specific class of monomorphisms $u:U\into X$, the \emph{normal monomorphisms} \cite{B1}, see next section ;\\
(4) there is an intrinsic notion of abelian object \cite{B1}, see section \ref{abob};\\
(5) any reflexive relation in $\EE$ is an equivalence relation, i.e. the category $\EE$ is a Mal'tsev one \cite{B2}.

So, according to Principle (1), a pointed protomodular category is characterized by the validity of the \emph{split short five lemma}. Generally, Principle (5) is not directly  exploited in $\Gp$; we shall show in Section \ref{asc} how importantly it works out inside a pointed protomodular category $\EE$. Pointed protomodular varieties of universal algebras are characterized in \cite{BJ}.

\subsection{Normal monomorphisms}

\begin{definition}{\rm \cite{B1}
In any category $\EE$, given a pair $(u,R)$ of a monomorphism $u:U\into X$ and an equivalence relation $R$ on $X$, the monomorphism $u$ is said to be {\em normal to $R$ }when the equivalence relation $u^{-1}(R)$ is the indiscrete equivalence relation $\nabla_X=R[\tau_X]$ on $X$ and, moreover, any commutative square in the following induced diagram is a pullback:}
$$\xymatrix@=3pt{
	{U\times U\;}   \ar@<-2ex>[ddd]_{d_0^U}  \ar@{>->}[rrrrr]^{\check u} \ar@<2ex>[ddd]^{d_1^U} &&&&&  {R\;}  \ar@<-2ex>[ddd]_{d_0^R} \ar@<2ex>[ddd]^{d_1^R} \\
	&&&&\\
	&&&&\\
	{U\;} \ar@{>->}[rrrrr]_{u}  \ar[uuu]|{s_0^U} &&&&& X \ar[uuu]|{s_0^R}
}
$$
\end{definition}
In the category $Set$, provided that $U\neq \emptyset$, these two properties characterize the equivalence classes of  $R$. By the Yoneda embedding, this implies the following:
\begin{proposition}
Given any equivalence relation $R$ on an object $X$ in a category $\EE$, for any map $x$, the following upper monomorphism  $\check x=d_1^R.\bar x$ is normal to $R$:
$$\xymatrix@=3pt{
	{I_R^x;}   \ar[ddd]  \ar@{>->}[rrrrr]^{\bar x}  &&&&&  {R\;}  \ar[ddd]_{d_0^R}  \ar[rrrrr]^{d_1^R}  &&&&& X\\
	&&&&\\
	&&&&\\
	{1\;} \ar@{>->}[rrrrr]_{x}   &&&&& X 
}
$$
\end{proposition}
In a pointed category $\EE$, taking the initial map $0_X: 1\into X$ gives rise to a monomorhism $\iota_R: I_R\into X$ which is normal to $R$. This construction produces a preorder mapping $\iota^X: \Equ_X\EE \to \Mon_X\EE$ from the preorder of the equivalence relations on $X$ to the preorder of subobjects of $X$ which preserves intersections. Starting with any map $f: X\to Y$, we get $I_{R[f]}=\Ker f$ which says that any kernel map $k_f$ is  normal to $R[f]$. Principle (3) above is a consequence of the fact \cite{B1} that in a protomodular category a monomorphism is normal to at most one equivalence relation (up to isomorphism). So that being normal, for a monomorphism $u$, becomes a property in this kind of categories. This is equivalent to saying that the preorder homomorphism $\iota_X: \Equ_X\EE \to \Mon_X\EE$ reflects inclusions; so, the preorder $\Norm_X$ of normal subobjects of $X$ is just the image $\iota^X(\Equ_X)\subset \Mon_X$.

\subsection{Regular context}

Let us recall from \cite{Barr} the following:
\begin{definition} {\rm 
A category $\EE$ is {\em regular }when it satisfies the two first conditions, and {\em exact} when it satisfies all the three conditions:\\
(1) regular epimorphisms are stable under pullbacks;\\
(2) any kernel equivalence relation $R[f]$ has a quotient $q_f$;\\
(3) any equivalence relation $R$ is a kernel equivalence relation.}
\end{definition}
Then, in the regular context, given any map $f: X\to Y$, the following canonical factorization $m$ is necessarily a monomorphism:
$$\xymatrix@=3pt{
	&&&  \Imm_f  \ar@{ >.>}[dddrrr]^{m} \\
	&&&&\\
	&&&&\\
	{X\;} \ar@{->>}[rrruuu]^{q_f} \ar[rrrrrr]_{f}   &&&&&& Y 
}
$$
This produces a canonical decomposition of the map $f$ in a monomorphism and a regular epimorphism which is stable under pullbacks. Now, given any regular epimorhism $f:X\onto Y$ and any subobject $u:U\into X$, the {\em direct image} $f(u): f(U)\into Y$ of $u$ along the regular epimorphism $f$ is given by $f(U)=\Imm_{f.u}\into Y$.

Any variety in the sense of Universal Algebra is exact and regular epimorphisms coincide with surjective homomorphisms.

\subsection{Homological categories}

The significance of pointed protomodular categories grows up in the regular context since, in this context, the split short five lemma can be extended to any exact sequence. Furthermore, the $3\times 3$ lemma, Noether isomorphisms and snake lemma hold; they are all collected in \cite{BB}. This is the reason why a regular pointed protomodular category $\EE$ is called \emph{homological}.

\section{Protomodular aspects of skew braces}

\subsection{Digroups}

From \cite{Normal}, we get the characterization of normal monomorphisms in $\DiGp$:
\begin{proposition}\label{normal1}
	A suboject $i: (G,*,\circ)\into (K,*,\circ)$ is normal in the category $\DiGp$  if and only if the three following conditions hold:\\
	{\rm (1)} $i:(G,*)\into (K,*)$ is normal in $\Gp$,\\{\rm (2)} $i:(G,\circ)\into (K,\circ)$ is normal in $\Gp$,\\
	{\rm (3)} for  all $(x,y)\in K\times K$, $x^{-*}*y\in G$ if and only if $x^{-\circ}\circ y\in G$.
\end{proposition}

\subsection{Skew braces}

The following observation is very important:
\begin{proposition}
Let $(G,*,\circ)$ be any skew brace. Consider the mapping $\lambda:G\times G\to G$	defined by $\lambda(a,u)=a^{-*}*(a\circ u)$. Then:\\
{\rm (1)} $\lambda_{a}=\lambda(a,-)$ is underlying a group homomorphism $(G,\circ)\to \Aut (G,*)$, and this condition is equivalent to {\rm (\ref{lsb})};\\
{\rm (2)} we have \begin{equation}\lambda({a^{-\circ}},u)=(a^{-\circ})^{-*}*(a^{-\circ}\circ u)=a^{-\circ}\circ(a*u).\label{doublech}\end{equation}
\end{proposition}
\proof
For (1), see \cite{GV}. For (2), we have $(a^{-\circ}\circ a)*(a^{-\circ})^{-*}*(a^{-\circ}\circ u)=
(a^{-\circ})^{-*}*(a^{-\circ}*u)=\lambda({a^{-\circ}},u)$.
\endproof

\subsection{First properties of skew braces} 

The following observation is straightforward:
\begin{proposition}
	$\SKB$ is a Birkhoff subcategory of $\DiGp$.
\end{proposition}
This means that any subobject of a skew brace in $\DiGp$ is a skew brace and that, given any surjective homomorphism $f: X\onto Y$ in $\DiGp$, the digroup $Y$ is a skew brace as soon as so is $X$. In this way, any equivalence relation $R$ in $\DiGp$ on a skew brace $X$ actually lies in $\SKB$ since it determines a subobject  $R\subset X\times X$ in $\DiGp$ and, moreover, its quotient in $\SKB$ is its quotient in $\DiGp$. The first part of this last sentence implies that any normal subobject $u:U\into X$ in $\DiGp$ with $X\in \SKB$ is  normal  in $\SKB$.

We are now going to show that the normal subobjects in $\SKB$ coincide with the ideals of \cite{GV}.
\begin{proposition}\label{normal2}
	A subobject $i: (G,*,\circ )\into (K,*,\circ )$ is normal in the category $\SKB$  if and only if the three following conditions hold:\\
	$(1)$ $i: (G,*)\into (K,*)$ is normal in $\Gp$,\\
	$(2)$ $i: (G,\circ )\into (K,\circ )$ is normal in $\Gp$,\\
	$(3')$ $\lambda_x(G)=G$ for all $x\in K$.
\end{proposition}
\proof
Suppose (1) and (2). We are going to show $(3)\iff (3')$, with $(3)$ given in Proposition \ref{normal1}.\\
(i) $x^{-\circ}\circ y\in G\Rightarrow x^{-*}*y\in G$ if and only if $\lambda_{x}(G)\subset G$, setting $y=x\circ u, \; u\in G$.\\
(ii) from (\ref{doublech}):
 $x^{-*}*y\in G \Rightarrow x^{-\circ}\circ y\in G$ if and only if $\lambda_{x^{-\circ}}(G)\subset G$, setting $y=x*u, \; u\in G$.\\
Finally $\lambda_{x}(G)\subset G$ for all $x$ is equivalent to $\lambda_{x}(G)=G$.
\endproof
\begin{corollary}
	A subobject $i: (G,*,\circ )\into (K,*,\circ )$ is normal in the category $\SKB$  if and only if it is an ideal in the sense of \cite{GV}, namely is such that:\\
	1) $i: (G,\circ )\into (K,\circ )$ is normal, 2) $G*a=a*G$ for all $a\in K$, 3) $\lambda_{a}(G)\subset G$ for all $a\in K$.
\end{corollary}
\proof
Straightforward. 
\endproof

Being a variety in the sense of Universal Algebra, $\SKB$ is finitely cocomplete; accordingly it has binary sums (called coproducts as well). So, $\SKB$ is a semi-abelian category according to the definition introduced in \cite{JMT}:
\begin{definition}
{\rm 	A pointed category $\EE$ is said to be {\em semi-abelian} when it is protomodular, exact and has finite sums.}
\end{definition}
From the same \cite{JMT}, let us recall the following observation which explains the choice of the terminology: a pointed category $\EE$ is abelian if and only if both $\EE$ and $\EE^{op}$ are semi-abelian.

\subsection{Internal skew braces}

Given any category $\EE$, the notion of internal group, digroup and skew brace is straightforward, determining the categories $\Gp\EE$,  $\DiGp\EE$ and  $\SKB\EE$. Since $\Gp\EE$ is protomodular, so are the two others. An important case is produced with $\EE=Top$ the category of topological spaces. Although $Top$ is not a regular category, so is the category $\Gp Top$, the regular epimorphisms being the  open surjective homomorphisms. So $\Gp Top$ is homological but not semi-abelian. 

Now let $f: X\to Y$ be any map in $\DiGp Top$. Let us show that $R[f]$ has a quotient in $\DiGp Top$. Take its quotient $q_{R[f]}: X \onto Q_f$ in $\DiGp$, then endow $Q_f$ with the quotient topology with respect to $R[f]$; then $q_{R[f]}$ is an open surjective homomorphism since so is $U_0(q_{R[f]})$. Accordingly, a regular epimorphism in $\DiGp Top$ is again an open  surjective homomorphism. Moreover this same functor $U_0: \DiGp Top \to \Gp Top$ being left exact and reflecting the homeomorphic isomorphisms, it reflects the regular epimorphisms; so, these regular epimorphisms in $\DiGp Top$ are stable under pullbacks. Accordingly the category $\DiGp Top$ is regular. Similarly the category $\SKB Top$ is homological as well, without being semi-abelian. As any category of topological semi-abelian algebras,  both $\DiGp Top$ and $\SKB Top$ are finitely cocomplete, see \cite{BC}.

\section{Skew braces and their commutators}

\subsection{Protomodular aspects}

\subsubsection{Commutative pairs of subobjects, abelian objects}\label{abob}

Given any pointed category $\EE$, the protomodular axiom applies to the following specific downward pullback:
$$
\xymatrix@=20pt{
	X  \ar@{ >->}[rr]^{r_X} \ar@<-4pt>[d]_{} && X\times Y \ar@<-4pt>[d]_{p_Y}  && \\
	1 \ar@{ >->}[rr]_{0_Y} \ar@{ >->}[u]_{0_{K}} & & Y \ar@{ >->}[u]_{l_Y} }
$$
where the monomorphisms are the canonical inclusions. This is the definition of a \emph{unital category} \cite{B2}. In this kind of categories there is an intrisic notion of \emph{commutative pair of subobjects}:
\begin{definition} {\rm 
	Let $\EE$ be a unital category.
	Given a pair $(u,v)$ of subobjects of $X$, we say that the subobjects $u$ and $v$ {\em cooperate} (or {\em commute}) when there is a (necessarily unique) map $\varphi $, called the {\em cooperator} of the pair $(u,v)$, making the following diagram commute:
	$$
	\xymatrix@=20pt{
		&  U \ar@{ >->}[dl]_{l_U}  \ar@{ >->}[dr]^{u}  & & \\
		U \times V \ar@{.>}[rr]_{\varphi} &&  X   \\
		& V \ar@{ >->}[ul]^{r_V}  \ar@{ >->}[ur]_{v}  & & 
	}
	$$
	 We denote this situation by $[u,v]=0$. A subobject $u:U\into Y$ is {\em central} when $[u,1_{X}]=0$. An object $X$ is {\em commutative}  when $[1_{X},1_{X}]=0$.}
\end{definition}
Clearly $[1_{X},1_{X}]=0$ gives $X$ a structure of internal unitary magma, which, $\EE$ being unital, is necessarily underlying an internal commutative monoid structure. When $\EE$ is protomodular, this is actually an internal abelian group structure, so that we call $X$ an abelian object \cite{B1}. This gives rise to a fully faithful subcategory $Ab(\EE)\hookrightarrow  \EE$, which is additive and stable under finite limits in $\EE$. From that we can derive:
\begin{proposition}\cite{B1}
A pointed protomodular category $\EE$ is additive if and only if any monomorphism is normal.
\end{proposition}

\subsubsection{Connected pairs $(R,S)$ of equivalence relations}

Since a protomodular category is necessarily a Mal'tsev one, we can transfer the following notions. Given any pair $(R,S)$ of equivalence relations  on the  object $X$ in $\EE$, take the following rightward and downward pullback:
$$
\xymatrix@=30pt
{ 
	R {\overrightarrow\times}_{\!\! X} S \ar[r]^{p_S} \ar[d]_{p_R}  &  S  \ar[d]_{d_0^S}   \ar@<+1,ex>@{ >->}[l]^{r_S} \\
	R \ar[r]^{d_1^R} \ar@<-1,ex>@{ >->}[u]_{l_R} & X \ar@<+1,ex>@{ >->}[l]^{s_0^R} \ar@<-1,ex>@{ >->}[u]_{s_0^S}
}
$$
where $l_R$ and $r_S$ are the sections induced by the maps $s_0^R$ and $s_0^S$. Let us recall the following definition from \cite{BG1}:
\begin{definition} {\rm 
	In a Mal'tsev category $\EE$, the pair $(R,S)$ is said to be {\em connected} when there is a (necessarily unique) morphism
	$$
	p : R {\overrightarrow\times}_{\!\! X} S \rightarrow X,\; xRySz\mapsto p(xRySz)
	$$
	such that $pr_S=d_1^S$ and $pl_R=d_0^R$, namely such that the following identities hold: $p(xRySy)=x$ and $p(yRySz)=z$. This morphism $p$ is then called the \emph{connector} of the pair, and we denote the situation by $[R,S]=0$.}
\end{definition}
From \cite{BG2}, let us recall that:
\begin{lemma}\label{func}
Let $\EE$ be a Mal'tsev category, $f: X\to Y$ any map, $(R,S)$ any pair of equivalence relations on $X$, $(\bar R,\bar S)$ any pair of equivalence relations on $Y$ such that $R\subset f^{-1}(\bar R)$ and  $S\subset f^{-1}(\bar S)$. Suppose moreover that $[R,S]=0$ and $[\bar R,\bar S]=0$. Then the following diagram necessarily commutes:
$$
\xymatrix@=30pt
{ 
	R {\overrightarrow\times}_{\!\! X} S \ar[rr]^{\tilde f} \ar[d]_{p_{(R,S)}}  && \bar R {\overrightarrow\times}_{\!\! Y} \bar S  \ar[d]^{p_{(\bar R,\bar S)}}    \\
	X \ar[rr]_{f}  && Y 
}
$$
where $\tilde f$ is the natural factorization induced by $f^{-1}(\bar R)$ and  $S\subset f^{-1}(\bar S)$.
\end{lemma}

A pointed Mal'tsev category is necessarily unital. From \cite{BG1}, in any pointed Mal'sev category $\EE$, we have necessarily
\begin{equation}[R,S]=0 \;\; \Rightarrow \;\;\ [I_R,I_S]=0\label{h=s}\end{equation}
In this way, the ``Smith commutation" \cite{S1976} implies the ``Huq commutation" \cite{H}. 

\subsection{Huq=Smith}

The converse is not necessarily true, even if $\EE$ is pointed protomodular, see Proposition \ref{notsh} below. When this is the case, we say that $\EE$ satisfies the {\rm (Huq=Smith)} condition. Any pointed strongly protomodular category satisfies {\rm (Huq=Smith)}, see \cite{BG1}. {\rm (Huq=Smith)} is true for $\Gp$ by the following straighforward:
\begin{proposition}\label{SHGp}
	Let $(R,S)$ be a pair of equivalence relations  in $\Gp$ on the group $(G,*)$. The following conditions are equivalent:\\
	{\rm (1)} $[I_R,I_S]=0$;\\
	{\rm (2)} $p(x,y,z)=x*y^{-1}*z$ defines a group homomorphism $p: G\times G \times G \to G$;\\
	{\rm (3)} $[R,S]=0$. 
\end{proposition}
\begin{proposition}\label{notsh}
	The category $\DiGp$ of digroups does not satisfy {\rm (Huq=Smith)}. 
\end{proposition}
\proof
We can use the counterexample introduced in \cite{Normal} for another purpose. Start with an abelian group $(A,+)$ and an object $a$ such that $-a\neq a$. Then define $\theta: A\times A \to A\times A$ as the involutive bijection which leaves fixed any object $(x,y)$ except $(a,a)$ which is exchanged with $(-a,a)$. Then defined the group structure $(A\times A,\circ)$ on $A\times A$ as the transformed along $\theta$ of $(A\times A,+)$. So, we get:
$$ (x,z)\circ(x',z')=\theta(\theta(x,z)+\theta(x',z'))$$
Clearly we have $(a,a)^{-\circ}=(a,-a)$. Since the second projection $\pi:A\times A \to A$ is such that $\pi\theta=\pi$, we get a digroup homomorphism $\pi: (A\times A,+,\circ)\to (A,+,+)$ whose kernel map is, up to isomorphism, $\iota_A: (A,+,+)\into (A\times A,+,\circ)$ defined by $\iota(a)=(a,0)$. The commutativity of the law $+$ makes $[\iota_A,\iota_A]=0$ inside $\DiGp$. We are going to show that, however we do not have $[R[\pi],R[\pi]]=0$. If it was the case, according to the previous proposition and considering the images by $U_0$ and $U_1$ of the desired ternary operation, we should have, for any triple $(x,y)R[\pi](x',y)R[\pi](x",y)$:
$$(x,y)-(x',y)+(x",y)=(x,y)\circ(x',y) ^{-\circ}\circ(x",y)$$
namely $(x,y)\circ(x',y) ^{-\circ}\circ(x",y)=(x-x'+x",y)$.
Now take $y=a=x'$ and $a\neq x \neq -a$. Then we get:\\ $(x,a)\circ(a,a) ^{-\circ}\circ(x",a)=(x,a)\circ(a,-a)\circ(x",a)=(x+a,0)\circ(x",a)$\\ $=(x+a+x",a)$, if moreover $a\neq x" \neq -a$. Now, clearly we get $x+a+x"\neq  x-a+x"$ since $a\neq -a$.
\endproof
However we have the following very general observation:
\begin{proposition}
	Let $\EE$ be any pointed Mal'tsev satisfying {\rm (Huq=Smith)}. So is any functor category $\mathcal F(\CC,\EE)$. 
\end{proposition}
\proof
Let $(R,S)$ be a pair of equivalence relation on an object $F\in F(\CC,\EE)$. We have $[R,S]=0$ if and only if for each object $C\in \CC$ we have $[R(C),S(C)]=0$ since, by Lemma \ref{func}, the naturality follows. In the same way, if $(u,v)$ is a pair of subfunctors of $F$, we have $[u,v]=0$ if and only if for each object $C\in \CC$ we have $[u(C),v(C)]=0$. Suppose now that $\EE$ satisfies {\rm (Huq=Smith)}, and that $[I_R,I_S]=0$. So,  for each object $C\in \CC$ we have $[I_R(C),I_S(C)]=0$, which implies $[R(C),S(C)]=0$. Accordingly $[R,S]=0$.
\endproof
Let $\TT$ be any finitary algebraic theory, and denote by $\TT(\EE)$ the category of internal $\TT$-algebras in $\EE$. Let us recall that, given any variety of algebras $\VV(\TT)$, we have a {\em Yoneda embedding for the internal $\TT$-algebras}, namely a left exact fully faithful factorization of the Yoneda embedding for $\EE$:
\[\xymatrix@C=2pc@R=2pc{ \TT(\EE) \ar@{-->}[rr]^{\bar Y_{\TT}} \ar[d]_{\mathcal U_{\TT}} && \mathcal F(\mathbb E^{op},\VV(\TT)) \ar[d]^{\mathcal F( \mathbb E^{op}, \mathcal U)} \\
	\EE \ar[rr]_Y && \mathcal F(\mathbb E^{op},Set)}\]
where $\mathcal U: \VV(\TT) \to Set$ is the canonical forgetful functor.
\begin{theorem}\label{TTE}
	Let $\TT$ be any finitary algebraic theory
	such that the associated variety of algebras $\VV(\TT)$ is pointed protomodular. If $\VV(\TT)$ satisfies {\rm (Huq=Smith)}, so does
	any category $\TT(\EE)$. 
\end{theorem}
\proof
If $\VV(\TT)$ satisfies {\rm (Huq=Smith)}, so does $\mathcal F(\mathbb E^{op},\VV(\TT))$ by the previous proposition. Accordingly, $\bar Y_{\TT}$ being left exact and fully faithful, so does $\TT(\EE)$.
\endproof

\subsection{Any category $\SKB\EE$ does satisfy {\rm (Huq=Smith)}}

\begin{proposition}\label{U,V}
	Given any pair $(U,V)$ of subobjects of $X$ in $\SKB$, the following conditions are equivalent:\\
	{\rm (1)} $[U,V]=0$;\\
	{\rm (2)} for all $(u,v)\in U\times V$, we get $u\circ v=u*v$ and this restriction is commutative;\\
	{\rm (3)} for all $(u,v)\in U\times V, \; \lambda_u(v)=v$, $[U_0(U),U_0(V)]=0$ and $[U_1(U),U_1(V)]=0$.\\
	Accordingly, an abelian object in $\SKB$ is necessarily of the form $(A,+,+)$ with $(A,+)$ abelian.
\end{proposition} 
\proof
Straightforward, setting $\varphi(u,v)=u+v$ and using an Eckmann-Hilton argument.
\endproof
 
\begin{proposition}[$\SKB$ does  satisfy {\rm (Huq=Smith)}]
	Let $R$ and $S$ be two equivalence relations on an object $X\in \SKB$. The following conditions are equivalent:\\
	{\rm (1)} $[I_R,I_S]=0$;\\
	{\rm (2)} $[U_0(U),U_0(V)]=0$, $[U_1(U),U_1(V)]=0$ and 
	$x*y^{-*}*z=x\circ y^{-\circ}\circ z$ for all $xRySz$;\\
	{\rm (3)} $[R,S]=0$. 
\end{proposition}
\proof
The identity  $x*y^{-*}*z=x\circ y^{-\circ}\circ z$ is equivalent to
$$y^{-\circ}\circ z=x^{-\circ}\circ(x*y^{-*}*z)=(x^{-\circ}\circ x)*(x^{-\circ})^{-*}*(x^{-\circ}\circ(y^{-*}*z))=$$ $$\qquad\qquad\qquad=(x^{-\circ})^{-*}*(x^{-\circ}\circ(y^{-*}*z)),$$ which, in turn, is equivalent to $$\lambda_{x^{-\circ}}(y^{-*}*z)=y^{-\circ}\circ z.$$

\smallskip

Suppose $xRySy$. Setting $z=y*v,\; v\in I_S$,  this is equivalent to $\lambda_{x^{-\circ}}(v)=y^{-\circ}\circ (y*v)=\lambda_{y^{-\circ}}(v)$ by (\ref{doublech}). This in turn is equivalent to $\lambda_{y}\circ \lambda_{x^{-\circ}}(v)=\lambda_{y\circ x^{-\circ}}(v)=v$, $v\in I_S$.
Setting $y=u \circ x ,\; u\in I_R$, this is equivalent to $\lambda_u(v)=v$, $(u,v)\in I_R\times I_S$.

\smallskip

Now, by Proposition \ref{U,V}, $[I_R,I_S]=0$ is equivalent to: for all $(u,v)\in I_R\times I_S$, we get $\lambda_u(v)=v$, $[U_0(U),U_0(V)]=0$ and $[U_1(U),U_1(V)]=0$. So we get  $[1) \iff 2)]$.

\smallskip

Suppose (2). From $[U_0(U),U_0(V)]=0$, we know by Proposition \ref{U,V} that $p(x,y,z)=x*y^{-*}*z$ is a group homomorphism  $(R {\overrightarrow\times}_X S,*),\to (X,*)$, and from $[U_1(U),U_1(V)]=0$ that $q(x,y,z)=x\circ y^{-\circ}\circ z$ is a group homomorphism $(R {\overrightarrow\times}_X S,\circ)\to (X,\circ)$. If $p=q$, this produces the desired $R {\overrightarrow\times}_{\!\! X} S \to X$ in $\SKB$ showing that $[R,S]=0$. Whence $[(2)\Rightarrow (3)]$. We have already noticed that the last implication $[(3)\Rightarrow (1)]$ holds in any pointed category.
\endproof
According to Theorem \ref{TTE}, we get the following:
\begin{corollary}\label{SHGpE}
	Given any category $\EE$, the category $\SKB\EE$ satisfies {\rm (Huq= Smith)}. This is the case in particular for the category $\SKB Top$ of topological skew  braces.
\end{corollary}

\subsection{Homological aspects of commutators}

\subsubsection{Abstract Huq commutator}

Suppose now that $\EE$ is any finitely cocomplete regular unital category. 
In this setting, we gave in \cite{B10}, for any pair $u: U \rightarrowtail X$, $v : V \rightarrowtail X$ of subobjects, the construction of a regular epimorphism $\psi_{(u,v)}$ which universally makes their direct images cooperate. Indeed consider the following diagram, where $Q[u,v]$ is the limit of the plain arrows:
$$
\xymatrix@=25pt{
	&  U \ar@{ >->}[dl]_{l_U}  \ar@{ >->}[dr]^{v}  & & \\
	U \times V \ar@{.>}[r]_{\bar{\psi}_{(u,v)}} & Q[u,v]  &  X \ar@{.>}[l]^{\psi_{(u,v)}}  \\
	& V \ar@{ >->}[ul]^{r_V} \ar@{ >->}[ur]_{v}  & & 
}
$$
The map $\psi_{(u,v)}$ is necessarily a regular epimorphism and the map $\bar{\psi}_{(u,v)}$ induces the cooperator of the direct images of the pair $(u,v)$ along  $\psi_{(u,v)}$. This regular epimorphism  $\psi_{(u,v)}$ measures the lack of cooperation of the pair $(u,v)$ in the sense that the map $\psi_{(u,v)}$ is an isomorphism if and only if $[u,v]=0$. We then get a symmetric tensor product: $I_{R[\psi_{(-,-)}]}: \Mon_X\times \Mon_X \to \Mon_X$ of preordered sets.

Since the map $\psi_{(u,v)}$ is a regular epimorphism, its distance from being an isomorphism is its distance from being a monomorphism, which is measured by the kernel equivalence relation $R[\psi_{(u,v)}]$. Accordingly,  in the homological context, it is meaningful to introduce the following definition, see also \cite{MM}:
\begin{definition}
	Given any finitely cocomplete homological category $\EE$ and any pair $(u,v)$ of subobjects of $X$, their abstract Huq commutator {\rm $[u,v]$} is defined as $I_{R[\psi_{(u,v)}]}$ or equivalently as the kernel map $k_{\psi_{(u,v)}}$.
\end{definition}
By this universal definition, in the category $\Gp$, this $[u,v]$ coincides with the usual $[U,V]$. 

\subsubsection{Abstract Smith commutator}\label{asc}

Suppose $\EE$ is a regular category. Then, given any regular epimorphism $f: X\onto Y$ and any equivalence relation $R$ on $X$, the direct image $f(R) \into Y\times Y$ of  $R\into X\times X$ along the regular epimorphism $f\times f: X\times X \onto Y\times Y$ is reflexive and symmetric, but generally not transitive. Now, when $\EE$ is a regular Malt'sev category, this direct image $f(R)$, being a reflexive relation, is an equivalence relation. 

Suppose moreover that $\EE$ is finitely cocomplete. Let $(R,S)$ be a pair of equivalence relations on $X$,
and consider the following diagram, where $Q[R,S]$ is the colimit of the plain arrows:
$$
\xymatrix@=25pt{
	& R \ar@{ >->}[dl]_{l_R} \ar[dr]^{d_{0,R}}  & & \\
	R \times_X S \ar@{.>}[r]_{\bar{\chi}_{(R,S)}} & Q[R,S]  &  X \ar@{.>}[l]^{\chi_{(R,S)}}  \\
	& S \ar@{ >->}[ul]^{r_S}  \ar[ur]_{d_{1,S}}  & & 
}
$$
Notice that, here, in consideration of the pullback defining $R \overrightarrow{\times}_X S$, the role of the projections $d_0$ and $d_1$ have been interchanged. This map $\chi_{(R,S)}$ measures the lack of connection between $R$ and $S$, see \cite{B10}:
\begin{theorem}
	Let $\EE$ be a finitely cocomplete regular Mal'tsev category. Then the map $\chi_{(R,S)}$ is a regular epimorphism and is the universal one which makes the direct images $\chi_{(R,S)}(R)$ and $\chi_{(R,S)}(S)$ connected. The equivalence relations $R$ and $S$ are connected (i.e. $[R,S]=0$) if and only if $\chi_{(R,S)}$ is an isomorphism.
\end{theorem}
Since the map $\chi_{(R,S)}$ is a regular epimorphism, its distance from being an isomorphism is its distance from being a monomorphism, which is exactly measured by its kernel equivalence relation $R[\chi_{(R,S)}]$. Accordingly, we give the following definition:
\begin{definition}
	Let $\EE$ be any finitely cocomplete regular Mal'tsev category. Given any pair $(R,S)$ of equivalence relations on $X$, their abstract Smith commutator $[R,S]$ is defined as the kernel equivalence relation $R[\chi_{(R,S)}]$ of the map $\chi_{(R,S)}$.
\end{definition}
In this way, we define a symmetric tensor product $[-,-]=R[\chi_{(-,-)}]: \Eq_X\times \Eq_X \to \Eq_X$ of preorered sets. It is clear that, with this definition, we get $[R,S]=0$ in the sense of connected pairs if and only if $[R,S]=\Delta_X$ (the identity equivalence relation on $X$) in the sense of this new definition. This is coherent since $\Delta_X$ is effectively the $0$ of the preorder $\Eq_X$. Let us recall the following:

\begin{proposition}\label{dim}
	Let $\EE$ be a pointed regular Mal'tsev category. Let $f : X \onto Y$ be a regular epimorphism and $R$ an equivalence relation on $X$. Then the direct image $f(I_R)$ of  the normal subjobject $I_R$ along $f$ is  $I_{f(R)}$.
\end{proposition}

From that, we can assert the following:

\begin{proposition}
	Let $\EE$ be a finitely cocomplete homological category. Given any pair $(R,S)$ of equivalence relations on $X$, we have {\em $[I_R,I_S] \subset I_{[R,S]}$}.
\end{proposition}
\proof
From (\ref{h=s}), we get $$ [\chi_{(R,S)}(R),\chi_{(R,S)}(S)]=0 \;\; \Rightarrow \;\; [I_{\chi_{(R,S)}(R)},I_{\chi_{(R,S)}(S)}]=0  $$
By the previous proposition we have:  $$0=[I_{\chi_{(R,S)}(R)},I_{\chi_{(R,S)}(S)}]=[ \chi_{(R,S)}(I_R),\chi_{(R,S)}(I_S)].$$
Accordingly, by the universal property of the regular epimorphism $\psi_{(I_R,I_S)}$ we get a factorization:
$$
\xymatrix@=20pt{
	X \ar@{->>}[rr]^{\psi_{(I_R,I_S)}} \ar@{->>}[rrd]_{\chi_{(R,S)}}&& Q[I_R,I_S] \ar@{.>}[d]\\
	&& Q[R,S]
}
$$
which shows that $[I_R,I_S]\subset I_{[R,S]}$.
\endproof

\begin{theorem}
	In a finitely cocomplete homological category $\EE$ 
	the following conditions are equivalent:\\
	{\rm (1)} $\EE$ satisfies {\rm (Huq=Smith)};\\
	{\rm (2)} {\em $[I_R,I_S]= I_{[R,S]}$} for any pair $(R,S)$ of equivalence relations on $X$.
	
	Under any of these conditions, the regular epimorphisms $\chi_{(R,S)}$ and $\psi_{(I_R,I_S)}$ do coincide.
\end{theorem}
\proof
Suppose (2). Then $[I_R,I_S]=0$ means that $\psi_{(I_R,I_S)}$ is an isomorphism, so that $0=[I_R,I_S]= I_{[R,S]}$. In a homological category $I_{[R,S]}=0$ is equivalent to $[R,S]=0$. Conversely, suppose (1). We have to find a factorization:
$$
\xymatrix@=25pt{
	X \ar@{->>}[rr]^{\psi_{(I_R,I_S)}} \ar@{->>}[rrd]_{\chi_{(R,S)}}&& Q[I_R,I_S] \\
	&& Q[R,S] \ar@{.>}[u]
}
$$
namely to show that $[\psi_{(I_R,I_S)}(R),\psi_{(I_R,I_S)}(S)]=0$. By (1) this is  equivalent to $0=[I_{\psi_{(I_R,I_S)}(R)},I_{\psi_{(I_R,I_S)}(S)}]$, namely to $0=[\psi_{(I_R,I_S)}(I_R),\psi_{(I_R,I_S)}(I_S)]$ by Proposition \ref{dim}. This is true by the universal property of the regular epimorphism $\psi_{(I_R,I_S)}$.
\endproof

\subsection{Skew braces and their commutators}

Since the categories $\SKB$ and $\SKB Top$ are finitely cocomplete homological categories, all the results of the previous section concerning commutators do apply and, in particular, thanks to the {\rm (Huq=Smith)} condition, the two notions of commutator are equivalent. It remains now to make explicit the description of the Huq commutator.

\bigskip

We will determine a set of generators for the Huq commutator of two ideals in a skew brace:

\begin{proposition}\label{2.5} If $I$ and $J$ are two ideals of a left skew brace $(A,*,\circ)$, their Huq commutator $[I,J]$ is the ideal of $A$ generated by the union of the following three sets: \\ \noindent {\rm (1)} the set $\{\, i\circ j\circ(j\circ i)^{-\circ}\mid i\in I,\ j\in J\,\}$, (which generates the commutator $[I,J]_{(A,\circ)}$ of the normal subgroups $I$ and $J$ of the group $(A,\circ)$); \\ \noindent {\rm (2)} the set $\{\, i* j*(j*i)^{-*}\mid i\in I,\ j\in J\,\}$, (which generates the commutator $[I,J]_{(A,*)}$ of the normal subgroups $I$ and $J$ of the group $(A,*))$; and \\ \noindent {\rm (3)} the set $\{\,(i\circ j)*(i* j)^{-*}
\mid i\in I,\ j\in J\,\}$.\end{proposition}

\begin{proof} Assume that the mapping $\mu\colon I\times J\to A/K$, $\mu(i,j)=i*j*K$ is a skew brace morphism for some ideal $K$ of $A$. Then $$\begin{array}{l} (i\circ j)\circ K=(i\circ K)\circ(j\circ K)=(i* K)\circ(j*K)=\\ \qquad=\mu(i,1)\circ\mu(1,j)=\mu((i,1)\circ(1,j))=\mu(i,j)=\mu((1,j)\circ(i,1))=\\ \qquad=\mu((1,j)\circ\mu(i,1))=(j*K)\circ(i*K)=(j\circ K)\circ(i\circ K)=(j\circ i)\circ K.\end{array}$$ This proves that the set $(1)$ is contained in $K$. 

Similarly, $$\begin{array}{l}(i* j)* K=(i* K)*(j* K)=\mu(i,1)*\mu(1,j)=\mu((i,1)*(1,j))=\mu(i,j)=\\ \qquad=\mu((1,j)*(i,1))=\mu((1,j)*\mu(i,1))=(j*K)*(i*K)=(j*i)*K.\end{array}$$ Thus the set $(2)$ is contained in $K$. 

Also, $$\begin{array}{l}(i\circ j)*K=(i\circ j)\circ K=(i\circ K)\circ(j\circ K)=(i*K)\circ(j*K)= \\ \qquad=\mu(i,1))\circ\mu(1,j)=\mu((i,1)\circ(1,j))=\mu(i,j)=\mu((i,1)*(1,j))=\\ \qquad=\mu(i,1)*\mu(1,j)=(i*K)*(j*K)=(i*j)*K.\end{array}$$ Hence the set (3) is also contained in $K$.

Conversely, let $K$ be the ideal of $A$ generated by the union of the three sets. It is then very easy to check that he mapping $\mu\colon I\times J\to A/K$, $\mu(i,j)=i*j*K$ is a skew brace morphism.\end{proof}

It the literature, great attention has been posed in the study of product $I\cdot J$ of two ideals $I,J$ of a (left skew) brace $(A,*,\circ)$. This product is with respect to the product $\cdot$ in the brace $A$ defined, for every $x,y\in A$ by $x\cdot y=y^{-*}*\lambda_x(y)$. Then, for every $i\in I$ and $j\in J$, $i\cdot j=j^{-*}*\lambda_i(j)=j^{-*}*i^{-*}*(i\circ j)=(i*j)^{-*}*(i\circ j)$. Hence the ideal of $A$ generated by the set  $I\cdot J$ of all products $i\cdot j$ coincides with the ideal of $A$ generated by the set (3) in the statement of Proposition~\ref{2.5}. 

\bigskip

Clearly, for a left skew brace $A$, the Huq commutator $[I,J]$ is equal to the Huq commutator $[J,I]$. Also, $I\cdot J=(J\cdot I)^{-*}$, so that the left annihilator of $I$ in $(A,\cdot)$ is equal to the right annihilator of $I$ in $(A,\cdot)$. Moreover, the condition ``$I\cdot J=0$'' can be equivalently expressed as ``$J$ is contained in the kernel of the group homomorphism $\lambda|^I\colon (A,\circ)\to\Aut(I,*)$.

\begin{proposition} For an ideal $I$ of a left skew brace $A$, there is a largest ideal of $A$ that centralizes $I$ (the {\em centralizer} of $I$).\end{proposition}

\begin{proof} The zero ideal centralizes $I$ and the union of a chain of ideals that centralize $I$ centralizes $I$. Hence there is a maximal element in the set of all the ideals of $A$ that centralize $I$. Now if $J_1$ and $J_2$ are two ideals of $A$, then $J_1*J_2=J_1\circ J_2$ is the join of $\{J_1,J_2\}$ in the lattice of all ideals of $A$. Now $J_1$ centralizes $I$ if and only if (1) $J_1\subseteq C_{(A,*)}(I)$, the centralizer of the normal subgroup $I$ in the group $(A,*)$; (2) $J_1\subseteq C_{(A,\circ)}(I)$, the centralizer of the normal subgroup $I$ in the group $(A,\circ)$; and (3) $J_1$ is contained in the kernel of the group morphism $\lambda|^I\colon (A,\circ)\to\Aut(I,*)$, which is a normal subgroup of $(A,\circ)$. Similarly for $J_2$. Hence if both $J_1$ and $J_2$ centralize $I$, then $J_1*J_2\subseteq C_{(A,*)}(I)$, and $J_1\circ J_2\subseteq C_{(A,\circ)}(I)\cap\ker\lambda|^I$. Therefore $J_1*J_2=J_1\circ J_2$ centralizes $I$. It follows that the set of all the ideals of $A$ that centralize $I$ is a lattice. Hence the maximal element in the set of all the ideals of $A$ that centralize $I$ is the largest element in that set.\end{proof}

In particular, the centralizer of the improper ideal of a left skew brace $A$ is the {\em center} of $A$.

\bigskip


%


A description of the free left skew brace over a set $X$ is available, in a language very different from ours, in \cite{Orza}.

\end{document}